\documentclass{article}

\newtheorem{theorem}{Theorem}
\newtheorem{lemma}{Lemma}
\newenvironment{proof}[1][Proof]
	{\noindent\textbf{#1.} }
	{\hfill\rule{0.5em}{0.5em}}

\begin{document}

\title{Constructing the Primitive Roots of Prime Powers}
\author{Nathan Jolly}

\maketitle

\begin{abstract}
	We use only addition and multiplication to construct the
	primitive roots of $p^{k+1}$ from the primitive roots of $p^{k}$,
	where $p$ is an odd prime and $k \ge 2$.
\end{abstract}

\section{Introduction}
\label{sec:intro}

There is a well-known result (Lemma~\ref{lem:powersofg}) which gives a
construction for all the primitive
roots of a given positive integer $n$ in terms of any one of them. This
construction is based on exponentiation.

In contrast, Niven et al.~take a different approach to constructing primitive
roots. Their result (Lemma~\ref{lem:niven}) gives an explicit
construction of all the primitive roots of $p^{2}$ from the primitive roots of
$p$, where $p$ is any prime. This alleviates the need for taking a single
primitive root of $p^{2}$ and exponentiating it in order to get the remaining
primitive roots. It also has the added benefit of shedding light on how the
primitive roots of $p^{2}$ are arranged, in terms of the primitive roots of
$p$. Furthermore, every primitive root constructed requires only a single
addition and multiplication, so the construction is easy to do by hand.

In this article, we show how to construct the primitive roots of $p^{k+1}$
from the primitive roots of $p^{k}$, where $k \ge 2$, and $p$ is an odd
prime.%
\footnote{Powers of the even prime $2$ do not have primitive roots in general.
For nice proofs of this result, see \cite{ref:ore} and \cite{ref:romero}.}
The construction builds on the above result of Niven et al.~\cite{ref:niven},
and allows us to construct the primitive roots of arbitrary prime powers in
terms of the primitive roots of the primes.

\section{Preliminaries}
\label{sec:prelim}

Let $n$ be a positive integer, and suppose that $a$ is relatively prime to
$n$. The \emph{order of $a$ mod $n$} is the smallest positive integer $m$
such that $a^{m} \equiv 1 \pmod{n}$, and we write $m = \mathrm{ord}_{n}(a)$.
The following result is used throughout this article.

\begin{lemma}[\cite{ref:dudley,ref:niven}]\label{lem:important}
	If $a$ and $n$ are relatively prime, then $a^{k} \equiv 1 \pmod{n}$
	if and only if $\mathrm{ord}_{n}(a)$ divides $k$.
\end{lemma}

Now Euler's theorem says that $a^{\varphi(n)} \equiv 1 \pmod{n}$, where
$\varphi(n)$ denotes the Euler totient function. It follows from
Lemma~\ref{lem:important} that $\mathrm{ord}_{n}(a)$ is always a positive
divisor of $\varphi(n)$, and hence cannot exceed $\varphi(n)$.
In fact $\mathrm{ord}_{n}(a)$ can actually reach $\varphi(n)$, and if it
does then $a$ is called a \emph{primitive root of $n$}.
For more information about orders and primitive roots, see
\cite{ref:burton}, \cite{ref:dudley}, \cite{ref:niven},
\cite{ref:ore}, \cite{ref:romero}, and \cite{ref:rosen}.

It is easy to verify that every integer congruent to $a$ mod $n$ has the
same order as $a$, so it suffices to consider single representatives from
every congruence class of integers mod $n$ which are relatively prime to
$n$. Niven et al.~\cite{ref:niven} call any such set of representatives a
\emph{reduced residue system mod $n$}, and it is easy to verify that there
are always $\varphi(n)$ elements in any one of these sets. The reduced residue
system we will be using in this article is the standard one, given by the set
\begin{equation}
	\{ 1 \le a \le n : \gcd(a,n) = 1 \}.
	\label{eqn:reduced}
\end{equation}

We refer to the number of primitive roots found in the set
(\ref{eqn:reduced}) above as \emph{the number of primitive roots of $n$}.
It turns out that if $n$ has primitive roots, then the number of primitive
roots of $n$ is exactly $\varphi(\varphi(n))$. This is a direct consequence
of the following result.

\begin{lemma}[{\cite[Corollary~$8.4.1$]{ref:rosen}}]\label{lem:powersofg}
	If $g$ is a primitive root of $n$, then the primitive roots
	of $n$ are given by the numbers $g^{k}$, where $k$ is relatively
	prime to $\varphi(n)$.
\end{lemma}

Note that Lemma~\ref{lem:powersofg} is precisely the result alluded to in
Section~\ref{sec:intro}, which allows us to construct all the primitive roots
of $n$ in terms of one of them.

The only positive integers which have primitive roots are $1$, $2$, $4$,
$p^{k}$, and $2p^{k}$, where $p$ is any odd prime and $k \ge 1$. Together
with Lemma~\ref{lem:powersofg}, the following classic results allow us to
construct the primitive roots of all the numbers that have them in terms of
the primitive roots of the primes.

\begin{lemma}[\cite{ref:niven,ref:romero,ref:rosen}]\label{lem:classic}
	\begin{enumerate}
		\item\label{part:onerootofp2}
			If $g$ is a primitive root of a prime $p$, then at least one of
			(but not necessarily both of) $g$ and $g + p$ is a primitive root
			of $p^{2}$.
	
		\item\label{part:onerootofpk}
			If $p$ is any odd prime and $g$ is a primitive root of $p^{2}$,
			then $g$ is a primitive root of $p^{k}$ for all $k \ge 2$.
	
		\item\label{part:rootsof2pk}
			If $p$ is an odd prime, $k \ge 1$, and $g$ is a primitive root of
			$p^{k}$, then the odd number out of $g$ and $g + p^{k}$ is a
			primitive root of $2p^{k}$.
	\end{enumerate}
\end{lemma}

Part~\ref{part:onerootofp2} of Lemma~\ref{lem:classic} is proved in
\cite[page~179]{ref:romero} and \cite[Theorem~8.9]{ref:rosen},
Part~\ref{part:onerootofpk} is proved in \cite[page~102]{ref:niven}
and \cite[page~179]{ref:romero}, and Part~\ref{part:rootsof2pk} is proved
in \cite[Theorem~8.14]{ref:rosen}.
Note that Parts \ref{part:onerootofp2} and~\ref{part:onerootofpk}
only give partial constructions of the primitive roots of prime powers;
Lemma~\ref{lem:powersofg} then gives the remaining primitive roots.

\section{An alternate construction}
\label{sec:alt}

In this section we give an alternate construction by Niven et al.~of the
primitive roots of $p^{2}$. We also state the main theorem of this article,
which gives a construction of the primitive roots of $p^{k+1}$ for $k \ge 2$.
We illustrate both of these results with a running example.

\begin{lemma}[{Niven et al.~\cite[Theorem~$2.39$]{ref:niven}}]\label{lem:niven}
	If $p$ is prime and $g$ is a primitive root of $p$, then $g + tp$ is
	a primitive root of $p^{2}$ for all values $0 \le t \le p-1$ except one.
	This exceptional value of $t$ is given by the formula
	\begin{equation}
		t = \frac{1 - g^{p-1}}{p}\left((p-1)g^{p-2}\right)^{-1},
		\label{eqn:except}
	\end{equation}
	where $t$ is understood to be reduced mod $p$ if necessary, and
	$((p-1)g^{p-2})^{-1}$ denotes the multiplicative inverse of
	$(p-1)g^{p-2}$, mod $p$.
	Furthermore, all the primitive roots of $p^{2}$ can be constructed
	this way.
\end{lemma}

Let us use Lemma~\ref{lem:niven} to construct the primitive roots of
$3^{2} = 9$. It is easily verified that $2$ is the only (incongruent)
primitive root of the prime $3$, and applying Lemma~\ref{lem:niven}
gives that $2 + 3t$ is a primitive root of $3^{2}$ for all
$0 \le t \le 2$, except for
\[
	t	= \frac{1-2^{3-1}}{3}\left((3-1)2^{3-2}\right)^{-1}
		= \frac{-3}{3}(4)^{-1}
		= -1 \cdot 1
		= -1
		\equiv 2 \pmod{3}.
\]
Thus the numbers
\[
	2 + 3 \cdot 0 = 2 \qquad\mbox{and}\qquad 2 + 3 \cdot 1 = 5
\]
are primitive roots of $3^{2}$, and $2 + 3 \cdot 2 = 8$ is not a primitive
root. Furthermore, as $2$ is the only primitive root of $3$,
Lemma~\ref{lem:niven} implies that $2$ and $5$ are the only primitive
roots of $3^{2}$.

That was a very simple example; it is not at all hard to compute the primitive
roots of $3^{2} = 9$ directly. Thus we give a slightly more sophisticated
example: a construction of the primitive roots of $5^{2} = 25$.
It is easy to verify that $2$ and $3$ are the only primitive roots of the
prime $5$. By Lemma~\ref{lem:niven}, it follows that the numbers
\begin{eqnarray*}
	2 + 5t_{1} && (0 \le t_{1} \le 4) \\
	3 + 5t_{2} && (0 \le t_{2} \le 4)
\end{eqnarray*}
are primitive roots of $5^{2} = 25$ for all $0 \le t_{1},t_{2} \le 4$,
except for the values
\[
	t_{1}	= \frac{1-2^{5-1}}{5} \left((5-1)2^{5-2}\right)^{-1}
			= \frac{-15}{5} \cdot (32)^{-1}
			= -3 \cdot 3
			\equiv 1 \pmod{5}
\]
and
\[
	t_{2}	= \frac{1-3^{5-1}}{5} \left((5-1)3^{5-2}\right)^{-1}
			= \frac{-80}{5} \cdot (108)^{-1}
			= -16 \cdot 2
			\equiv 3 \pmod{5}.
\]
Thus the following numbers are all primitive roots of $5^{2}$:
\begin{eqnarray*}				
	2 + 5 \cdot 0 = 2 && 3 + 5 \cdot 0 = 3 \\
	2 + 5 \cdot 2 = 12 && 3 + 5 \cdot 1 = 8 \\
	2 + 5 \cdot 3 = 17 && 3 + 5 \cdot 2 = 13 \\
	2 + 5 \cdot 4 = 22 && 3 + 5 \cdot 4 = 23
\end{eqnarray*}
As $2$ and $3$ are the only primitive roots of $5$, it follows by
Lemma~\ref{lem:niven} that we have constructed all the primitive roots
of $5^{2}$.

The main theorem of this article extends the construction of primitive roots
given in Lemma~\ref{lem:niven} to arbitrary powers of an odd prime.

\begin{theorem}\label{thm:main}
	If $p$ is an odd prime, $k \ge 2$, and $g$ is a primitive root of $p^{k}$,
	then $g + tp^{k}$ is a primitive root of $p^{k+1}$ for all
	$0 \le t \le p-1$.
	Furthermore, all the primitive roots of $p^{k+1}$ can be constructed
	this way.
\end{theorem}

As an illustration of Theorem~\ref{thm:main}, let us construct the primitive
roots of $3^{3} = 27$ from the primitive roots of $3^{2} = 9$. We saw above
that the primitive roots of $3^{2}$ are $2$ and $5$. It follows by
Theorem~\ref{thm:main} that the numbers
\begin{eqnarray*}
	2 + 3^{2} \cdot 0 = 2 && 5 + 3^{2} \cdot 0 = 5 \\
	2 + 3^{2} \cdot 1 = 11 && 5 + 3^{2} \cdot 1 = 14 \\
	2 + 3^{2} \cdot 2 = 20 && 5 + 3^{2} \cdot 2 = 23
\end{eqnarray*}
are all primitive roots of $3^{3}$---no exceptions. Furthermore, as $2$ and $5$
are the only primitive roots of $3^{2}$, Theorem~\ref{thm:main} implies that
the numbers above make up all the primitive roots of $3^{3}$. This can be
verified by direct calculation.

Let us go one step further. We now know that the primitive roots of $3^{3}$
are $2$, $5$, $11$, $14$, $20$, and $23$. Applying Theorem~\ref{thm:main}
again, it follows that the numbers
\begin{eqnarray*}
2 + 3^{3} \cdot 0 = 2 & 2 + 3^{3} \cdot 1 = 29 & 2 + 3^{3} \cdot 2 = 56 \\
5 + 3^{3} \cdot 0 = 5 & 5 + 3^{3} \cdot 1 = 32 & 5 + 3^{3} \cdot 2 = 59 \\
11 + 3^{3} \cdot 0 = 11 & 11 + 3^{3} \cdot 1 = 38 & 11 + 3^{3} \cdot 2 = 65 \\
14 + 3^{3} \cdot 0 = 14 & 14 + 3^{3} \cdot 1 = 41 & 14 + 3^{3} \cdot 2 = 68 \\
20 + 3^{3} \cdot 0 = 20 & 20 + 3^{3} \cdot 1 = 47 & 20 + 3^{3} \cdot 2 = 74 \\
23 + 3^{3} \cdot 0 = 23 & 23 + 3^{3} \cdot 1 = 50 & 23 + 3^{3} \cdot 2 = 77
\end{eqnarray*}
are all primitive roots of $3^{4} = 81$. Furthermore, Theorem~\ref{thm:main}
implies that these are all the primitive roots of $3^{4}$. This can be checked
by a direct calculation.

Finally, to complete the picture we recall a method which allows us to
construct the primitive roots of $2p^{k}$ from $p^{k}$, where $p$ is any
odd prime and $k \ge 1$.

\begin{lemma}[{\cite[page~173]{ref:burton}}]\label{lem:2pk}
	Let $p$ be an odd prime, let $k \ge 1$, and suppose that $g$ is a primitive
	root of $p^{k}$. Then $g$ is a primitive root of $2p^{k}$ if $g$ is odd, and
	$g + p^{k}$ is a primitive root of $2p^{k}$ if $g$ is even. Furthermore,
	every primitive root of $2p^{k}$ can be constructed in this way.
\end{lemma}

From Lemma~\ref{lem:2pk} it is clear that all the odd primitive roots of
$3^{4}$ are primitive roots of $2 \cdot 3^{4} = 162$, and adding $3^{4}$ to
the even primitive roots of $3^{4}$ will give the remaining primitive roots
of $2 \cdot 3^{4}$. Thus the numbers
	\begin{eqnarray*}
		5 && 2 + 3^{4} = 83 \\
		11 && 14 + 3^{4} = 95 \\
		23 && 20 + 3^{4} = 101 \\
		29 && 32 + 3^{4} = 113 \\
		41 && 38 + 3^{4} = 119 \\
		47 && 50 + 3^{4} = 131 \\
		59 && 56 + 3^{4} = 137 \\
		65 && 68 + 3^{4} = 149 \\
		77 && 74 + 3^{4} = 155
	\end{eqnarray*}
are all primitive roots of $2 \cdot 3^{4}$, and Lemma~\ref{lem:2pk} implies
that there are no others.

\section{Hensel's Lemma}
\label{sec:hensel}

In order to complete the proof of Lemma~\ref{lem:niven}, Niven et al.~use a
result known as \emph{Hensel's Lemma}. Hensel's Lemma is used to show that
there is exactly one value in the range $0 \le t \le p-1$ such that $g + tp$
is not a primitive root of $p^{2}$ (where $g$ is a primitive root of the prime
$p$). Furthermore, Hensel's Lemma gives the \emph{formula} (\ref{eqn:except})
for this exceptional $t$, and it is hard to see how this formula could arise
by more direct means.

Hensel's Lemma originated in the theory of $p$-adic numbers, but it also has
an equivalent form in the theory of polynomial congruences. We will be using
the latter form. In this section we recall some basic
definitions from the theory of polynomial congruences, and we state Hensel's
Lemma. For more information about polynomial congruences
and the corresponding version of Hensel's Lemma, see \cite{ref:niven} and
\cite{ref:romero}.

Let $f(x)$ be a polynomial with integer coefficients. Then a
\emph{polynomial congruence} is a congruence of the form
\begin{equation}
	f(x) \equiv 0 \pmod{n},
	\label{eqn:genpolycon}
\end{equation}
where $n$ is called the {\em modulus} of the congruence.
An integer $x_{0}$ satisfying $f(x_{0}) \equiv 0 \pmod{n}$
is called a {\em solution} of the congruence~(\ref{eqn:genpolycon}).

In this article we will only be interested in polynomial congruences
of prime power moduli.
All of these congruences have the form
\begin{equation}
	f(x) \equiv 0 \pmod{p^{k}},
	\label{eqn:pk}
\end{equation}
where $p$ is prime and $k \ge 1$.
Now it turns out (see \cite{ref:romero} for details)
that all the solutions of
\begin{equation}
	f(x) \equiv 0 \pmod{p^{k+1}}
	\label{eqn:pk+1}
\end{equation}
can be constructed from the solutions of (\ref{eqn:pk}).
Specifically, if $x_{1}$ is a solution of (\ref{eqn:pk+1})
then we can always write
\begin{equation}
	x_{1} = x_{0} + tp^{k},
	\label{eqn:lift}
\end{equation}
where $x_{0}$ is some solution to (\ref{eqn:pk}) and the integer $t$
is yet to be determined.
Following \cite{ref:niven}, we say that the solution $x_{1}$ in
(\ref{eqn:lift}) \emph{lies above} the solution $x_{0}$, and that $x_{0}$
\emph{lifts to} $x_{1}$.

\begin{lemma}[Hensel's Lemma \cite{ref:niven,ref:romero}]
	Let $f(x)$ be a polynomial with integer coefficients, and suppose that $p$
	is prime and that $k \ge 1$. Let $x_{0}$ be a solution to
	$f(x) \equiv 0 \pmod{p^{k}}$.
	Then exactly one of the following occurs:
	\begin{enumerate}
		\item\label{item:one}
			If $f'(x_{0}) \not\equiv 0 \pmod{p}$ then $x_{0}$ lifts to exactly
			one solution $x_{1}$ to $f(x) \equiv 0 \pmod{p^{k+1}}$.
			This solution is given by $x_{1} = x_{0} + tp^{k}$, where
			\[
				t = -\frac{f(x_{0})}{p^{k}}\left(f'(x_{0})\right)^{-1}.
			\]
			Here $t$ is understood to be reduced mod $p$ if necessary, and
			$(f'(x_{0}))^{-1}$ stands for the multiplicative inverse of
			$f'(x_{0})$, mod $p$.

		\item\label{item:many}
			If $f'(x_{0}) \equiv 0 \pmod{p}$ and $x_{0}$ is a solution of
			$f(x) \equiv 0 \pmod{p^{k+1}}$, then $x_{0}$ lifts to
			$x_{1} = x_{0} + tp^{k}$ for all integers $0 \le t \le p - 1$.
			Thus $x_{0}$ lifts to $p$ distinct solutions of
			$f(x) \equiv 0 \pmod{p^{k+1}}$.

		\item\label{item:none}
			Finally, if $f'(x_{0}) \equiv 0 \pmod{p}$ but $x_{0}$ is not a
			solution of $f(x) \equiv 0 \pmod{p^{k+1}}$, then $x_{0}$ does not
			lift to any solutions of $f(x) \equiv 0 \pmod{p^{k+1}}$.
			Thus if $f(x) \equiv 0 \pmod{p^{k+1}}$ has any
			solutions at all, then they do not lie above $x_{0}$.
	\end{enumerate}
\end{lemma}

\section{Proof of the Main Theorem}
\label{sec:proof}

\begin{proof}[Proof of Theorem~\ref{thm:main}]
	First we show that for any $0 \le t \le p-1$, either $g + tp^{k}$ is a
	primitive root of $p^{k+1}$, or $\mathrm{ord}_{p^{k+1}}(g + tp^{k})
	= \varphi(p^{k})$. Set $h = \mathrm{ord}_{p^{k+1}}(g + tp^{k})$.
	Then $(g + tp^{k})^{h} \equiv 1 \pmod{p^{k+1}}$, and so
	$(g + tp^{k})^{h} \equiv 1 \pmod{p^{k}}$. Now $g \equiv g + tp^{k}
	\pmod{p^{k}}$, so
	\[
		g^{h} \equiv (g + tp^{k})^{h} \equiv 1 \pmod{p^{k}}.
	\]
	This means that $h$ is a multiple of $\mathrm{ord}_{p^{k}}(g)$,
	which is $\varphi(p^{k})$ since $g$ is a primitive root of $p^{k}$.
	Hence we can write
	\begin{equation}
		h = y\varphi(p^{k})
		\label{eqn:hmult}
	\end{equation}
	for some positive integer $y$.
	
	On the other hand, $h$ was defined to be the order of $g + tp^{k}$,
	mod $p^{k+1}$. Thus $h$ is a divisor of $\varphi(p^{k+1})
	= p\varphi(p^{k})$, so
	\begin{equation}
		p\varphi(p^{k}) = xh
		\label{eqn:hdiv}
	\end{equation}
	for some positive integer $x$.
	Putting Equations (\ref{eqn:hmult}) and~(\ref{eqn:hdiv}) together give that
	\[
		p\varphi(p^{k}) = xh = xy\varphi(p^{k}),
	\]
	which implies that
	\[
		xy = p.
	\]
	Since $p$ is prime and $y$ is positive, this implies that either $y = p$
	or $y = 1$.
	
	Together with (\ref{eqn:hmult}), this information gives that either
	$h = p\varphi(p^{k}) = \varphi(p^{k+1})$ or $h = \varphi(p^{k})$.
	In the former case we have that $g + tp^{k}$ is a primitive root of
	$p^{k+1}$, and in the latter case we have that
	$\mathrm{ord}_{p^{k+1}}(g + tp^{k}) = \varphi(p^{k})$.
	
	Thus, to prove Theorem~\ref{thm:main}, it suffices to show that
	$\mathrm{ord}_{p^{k+1}}(g + tp^{k}) \not= \varphi(p^{k})$
	for all $0 \le t \le p-1$.
	Suppose that there \emph{was} some $0 \le t \le p-1$ such that
	$\mathrm{ord}_{p^{k+1}}(g + tp^{k}) = \varphi(p^{k})$.
	Then
	\[
		(g + tp^{k})^{\varphi(p^{k})} \equiv 1 \pmod{p^{k+1}},
	\]
	so $g + tp^{k}$ is a solution of the congruence $f(x) \equiv 0
	\pmod{p^{k+1}}$, where $f(x) = x^{\varphi(p^{k})} - 1$.
	As $g$ is a primitive root of $p^{k}$, $g$ is a solution of
	$f(x) \equiv 0 \pmod{p^{k}}$, and hence $g$ lifts to $g + tp^{k}$.
	Now $f'(x) = x^{\varphi(p^{k})-1}$, so we have that
	\begin{eqnarray*}
		f'(g)	&=& \varphi(p^{k})g^{\varphi(p^{k}) - 1} \\
				&=& (p^{k} - p^{k - 1})g^{\varphi(p^{k}) - 1} \\
				&=& p(p^{k - 1} - p^{k - 2})g^{\varphi(p^{k}) - 1} \\
				&\equiv& 0 \pmod{p}.
	\end{eqnarray*}
	Notice that the above fails if $k = 1$. Thus it is imperative
	that we have $k \ge 2$.
	
	Finally, since $g$ is a primitive root of $p^{2}$, it follows from
	Part~(\ref{part:onerootofpk}) of Lemma~\ref{lem:classic}
	that $g$ is a primitive root of $p^{k+1}$. Hence
	\[
		g^{\varphi(p^{k})} \not\equiv 1 \pmod{p^{k+1}},
	\]
	since $\varphi(p^{k}) < \varphi(p^{k+1}) = \mathrm{ord}_{p^{k+1}}(g)$.
	Thus $g$ is not a solution of $f(x) \equiv 0 \pmod{p^{k+1}}$, so
	Part~(\ref{item:none}) of Hensel's Lemma applies.
	It follows that $g$ does not lift to any solutions of
	$f(x) \equiv 0 \pmod{p^{k+1}}$, which is a contradiction.
	
	Therefore $\mathrm{ord}_{p^{k+1}}(g + tp^{k}) \not= \varphi(p^{k})$
	for all $0 \le t \le p-1$, so $g + tp^{k}$ is a primitive root of
	$p^{k+1}$ for all $0 \le t \le p-1$.
	It only remains to show that we can construct all the primitive roots
	of $p^{k+1}$ in this way. As $t$ ranges from $0$ to $p-1$, $tp^{k}$
	runs through all the multiples of $p^{k}$ from $0$ to $(p-1)p^{k}$.
	As all of these are distinct mod $p^{k+1}$, it follows that the numbers
	$g + tp^{k}$ are distinct mod $p^{k+1}$ for all $0 \le t \le p-1$.
	Furthermore, if $g_{1}$ and $g_{2}$ are any primitive roots of
	$p^{k}$ and $0 \le t_{1},t_{2} \le p-1$, then
		$g_{1} + t_{1}p^{k} \equiv g_{2} + t_{2}p^{k} \pmod{p^{k+1}}$
	implies that
		$g_{1} + t_{1}p^{k} \equiv g_{2} + t_{2}p^{k} \pmod{p^{k}}$,
	and hence
	\[
		g_{1} - g_{2} \equiv (t_{1} - t_{2})p^{k} \equiv 0 \pmod{p^{k}},
	\]
	that is, $g_{1} \equiv g_{2} \pmod{p^{k}}$.
	
	It follows that as $g$ ranges through all the $\varphi(\varphi(p^{k}))$
	primitive roots of $p^{k}$, and as $t$ ranges independently from $0$
	to $p-1$, there will be $p\varphi(\varphi(p^{k}))$ distinct constructed
	primitive roots of $p^{k+1}$. As it can be verified that
	$p\varphi(\varphi(p^{k})) = \varphi(\varphi(p^{k+1}))$, it follows
	that we have constructed all the primitive roots of $p^{k+1}$
	exactly once.
\end{proof}

\bigskip\noindent
\textbf{Acknowledgements} The author thanks his supervisor Tom~E.~Hall
for excellent suggestions and advice, and Andrew~Z.~Tirkel and
Charles~F.~Osborne for encouraging him to write this article.

\noindent
\textbf{Nathan Jolly} is a mathematics {Ph.D.} student at Monash University,
Australia. The present article grew out of his honors thesis.
\emph{School of Mathematical Sciences, Monash University, Wellington Road,
Clayton, Victoria 3800, Australia}\\
\emph{email: nathan.jolly@sci.monash.edu.au}

\end{document}